\documentclass{amsart}

\usepackage[mathscr]{eucal}
\usepackage{amssymb}
\usepackage[usenames,dvipsnames]{color}
\usepackage[normalem]{ulem}
\usepackage{amsthm}
\usepackage{bbold}
\usepackage{enumerate}
\usepackage{array}
\usepackage{amsmath}
\usepackage{stmaryrd}

\usepackage[colorlinks=true,linkcolor={Brown},citecolor={Brown},urlcolor={Brown}]{hyperref}

\numberwithin{equation}{section}
\usepackage[all]{xy}
\SelectTips{cm}{}
\newdir{ >}{{}*!/-10pt/\dir{>}}

\swapnumbers 

\newtheorem{Thm}[equation]{Theorem}
\newtheorem*{Thm*}{Theorem}
\newtheorem{Prop}[equation]{Proposition}
\newtheorem{Lem}[equation]{Lemma}
\newtheorem{Cor}[equation]{Corollary}

\theoremstyle{remark}
\newtheorem{Def}[equation]{Definition}

\newtheorem{Exa}[equation]{Example}
\newtheorem{Hyp}[equation]{Hypotheses}
\newtheorem{Rem}[equation]{Remark}

\newcommand{\nc}{\newcommand}
\nc{\dmo}{\DeclareMathOperator}

\dmo{\Ab}{Ab}
\dmo{\Der}{D}
\dmo{\Ext}{Ext}
\dmo{\Hom}{Hom}
\dmo{\Id}{Id}
\dmo{\Ker}{Ker}
\dmo{\Loc}{Loc}
\dmo{\Mod}{Mod}
\dmo{\SH}{SH}
\dmo{\Spc}{Spc}
\dmo{\Tor}{Tor}
\dmo{\coker}{coker}
\dmo{\hocolim}{hocolim}
\dmo{\img}{im}
\dmo{\incl}{incl}
\dmo{\modname}{mod}%
\dmo{\opname}{op}
\dmo{\yonedaname}{h}

\nc{\Greg}[1]{{\color{CarnationPink}#1}}
\nc{\Henning}[1]{{\color{Blue}#1}}
\nc{\Homcat}[1]{\Hom_{\cat #1}}
\nc{\MMod}{\Mod\text{-}}%
\nc{\MT}{\MMod\cat{T}^{c}}
\nc{\Paul}[1]{{\color{Violet}#1}}
\nc{\Pout}[1]{{\color{red}\sout{#1}}}
\nc{\SET}[2]{\big\{\,#1\,\big|\,#2\,\big\}}
\nc{\Sout}[1]{\Paul{\sout{#1}}}
\nc{\adjto}{\rightleftarrows}
\nc{\adj}{\dashv}
\nc{\aka}{{a.\,k.\,a.}\ }
\nc{\bbN}{\mathbb{N}}
\nc{\bbP}{\mathbb{P}}
\nc{\bbZ}{\mathbb{Z}}
\nc{\bbe}{\mathbb{e}}
\nc{\bbf}{\mathbb{f}}
\nc{\bigjoin}{\bigvee}
\nc{\bigmeet}{\bigwedge}
\nc{\calI}{\mathcal{I}}
\nc{\calJ}{\mathcal{J}}
\nc{\cat}[1]{\mathscr{#1}}
\nc{\colim}{\mathop{\mathrm{colim}}}
\nc{\eg}{{\sl e.g.}}
\nc{\eps}{\epsilon}
\nc{\equalby}[1]{\overset{\textrm{#1}}=}
\nc{\fp}{^{\textrm{\rm fp}}}
\nc{\hook}{\hookrightarrow}
\nc{\ie}{{\sl i.e.}\ }
\nc{\ihom}{{\mathsf{hom}}} 
\nc{\into}{\mathop{\rightarrowtail}}
\nc{\inv}{^{-1}}
\nc{\isoto}{\overset{\sim}{\,\to\,}}
\nc{\join}{\vee}
\nc{\loccit}{{\sl loc.\ cit.}}
\nc{\mT}{\mmod\cat T^c}
\nc{\meet}{\wedge}
\nc{\mmod}{\modname\text{-}}%
\nc{\mutmut}{{\sl mutatis mutandis}}
\nc{\ointo}[1]{\,\overset{#1}\into\,}
\nc{\onto}{\mathop{\twoheadrightarrow}}
\nc{\oonto}[1]{\,\overset{#1}\onto\,}
\nc{\op}{^{\opname}}
\nc{\otoo}[1]{\overset{#1}{\,\too\,}}
\nc{\oto}[1]{\overset{#1}\to}
\nc{\oursetminus}{\!\smallsetminus\!}
\nc{\pperp}{{\perp\kern-.5em\perp}}
\nc{\qquadtext}[1]{\qquad\textrm{#1}\qquad}
\nc{\quadtext}[1]{\quad\textrm{#1}\quad}
\nc{\restr}[1]{_{|_{\scriptstyle #1}}}
\nc{\smat}[1]{\left(\begin{smallmatrix} #1 \end{smallmatrix}\right)}
\nc{\too}{\mathop{\longrightarrow}\limits}
\nc{\unit}{\mathbb{1}}
\nc{\yoneda}{\yonedaname}

\begin{document}


\title{The frame of smashing tensor-ideals}
\author{Paul Balmer}
\author{Henning Krause}
\author{Greg Stevenson}
\date{\today}

\address{Paul Balmer, Mathematics Department, UCLA, Los Angeles, CA 90095-1555, USA}
\email{balmer@math.ucla.edu}
\urladdr{http://www.math.ucla.edu/$\sim$balmer}

\address{Henning Krause, Universit\"at Bielefeld, Fakult\"at f\"ur Mathematik,
Postfach 10\,01\,31, 33501 Bielefeld, Germany}
\email{hkrause@math.uni-bielefeld.de}
\urladdr{http://www.math.uni-bielefeld.de/$\sim$hkrause/}

\address{Greg Stevenson, School of Mathematics and Statistics,
University of Glasgow,
University Place,
Glasgow G12 8QQ}
\email{gregory.stevenson@glasgow.ac.uk}
\urladdr{http://www.maths.gla.ac.uk/$\sim$gstevenson/}

\begin{abstract}
We prove that every flat tensor-idempotent in the module category $\MT$ of a tensor-triangulated category~$\cat T$ comes from a unique smashing ideal in~$\cat T$. We deduce that the lattice of smashing ideals forms a frame.
\end{abstract}


\thanks{P.\,Balmer supported by Humboldt Research Award and NSF grant~DMS-1600032.}

\maketitle

\begin{center}
\textit{Dedicated to Amnon Neeman, on the occasion of his 60${}^\textrm{th}$ birthday.}
\end{center}
\bigbreak
\tableofcontents


\section{Introduction}


The ultimate goal of this paper is to prove Theorem~\ref{thm:frame}, which says:
\begin{Thm}
\label{thm:frame-intro}%
Let $\cat T$ be a tensor-triangulated category, assumed to be rigidly-compactly generated. Then the lattice of smashing $\otimes$-ideals of~$\cat T$ is a \emph{frame}, \ie it is complete (all meets and joins exist) and the meet distributes over arbitrary joins.
\end{Thm}

Let us explain the statement and put it in perspective. Our triangular terminology is standard, following~\cite{BalmerFavi11} for instance, and is also recalled in~\ref{pt:remind-big-tt} below.

\smallbreak

In homotopy theory, algebraic geometry, representation theory and beyond, we encounter many `big' tensor-triangulated categories~$\cat T$ as in Theorem~\ref{thm:frame-intro}. Explicit examples include what Hovey, Palmieri and Strickland call `unital algebraic stable homotopy theories'; see~\cite[\S\,1.2]{HoveyPalmieriStrickland97}. Ever since Freyd~\cite{Freyd70}, we know that such big categories~$\cat T$ can really be \emph{wildly big} and we therefore approach them via the more concrete subcategory~$\cat T^c$ of rigid and compact objects -- also known as the `small' or `finite' objects. The interplay between big~$\cat T$ and small~$\cat T^c$ is a rich and fruitful topic. For instance, the famous Telescope Conjecture prophesies in some cases a bijection between thick $\otimes$-ideals $\cat J$ of~$\cat T^c$ and smashing $\otimes$-ideals $\cat S$ of~$\cat T$.

On the `small' side of things, our understanding of the lattice of thick $\otimes$-ideals $\cat J\subseteq\cat T^c$ has progressed substantially over the last three decades, starting with early work of Hopkins~\cite{Hopkins87} and others. In its modern general form, this theory goes by the name of \emph{tensor-triangular geometry}, for which we refer to~\cite{Balmer05a,BalmerICM} or~\cite{Stevenson16pp}. The foundation of tt-geometry involves a topological space, $\Spc(\cat T^c)$, called the \emph{spectrum} of~$\cat T^c$. An early result in the subject~\cite[Thm.\,4.10]{Balmer05a} states that the lattice of thick $\otimes$-ideals $\cat J\subseteq\cat T^c$ is isomorphic to the lattice of open subsets of~$\Spc(\cat T^c)$ with its Hochster-dual topology. This space $\Spc(\cat T^c)$ also carries so-called `supports' for objects of~$\cat T^c$ and allows us to follow geometric intuition.

Unfortunately, on the `big' side of the story, we lack an analogous geometric understanding of the lattice of smashing $\otimes$-ideals of~$\cat T$. No reasonable `big spectrum' for~$\cat T$ has been discovered yet. This hypothetical topological space controlling the smashing $\otimes$-ideals would open the door to what we could call `big tt-geometry', \ie a geometric study of the big tensor-triangulated category~$\cat T$ in the same way that $\Spc(\cat T^c)$ underlies the tt-geometry of the essentially small category~$\cat T^c$.

It is therefore desirable to understand how far the lattice of smashing $\otimes$-ideals of~$\cat T$ is from being \emph{spatial}, \ie in bijection with the lattice of open subsets of some topological space. Our Theorem~\ref{thm:frame-intro} makes a critical step in this direction. A complete lattice is a \emph{frame} if the meet operation (infimum)~$\meet$ distributes over arbitrary joins (supremum)~$\join$ as follows:
\begin{equation}
\label{eq:frame}%
A \meet \big(\,\bigjoin_{i\in I} B_i\big) \ = \ \bigjoin_{i\in I} (A \meet B_i)\,.
\end{equation}
This distributivity is a non-trivial condition on a lattice which is of course necessary for being spatial: just read $\meet$ as intersection and $\join$ as union in~\eqref{eq:frame}. Theorem~\ref{thm:frame-intro} is therefore a necessary property for the existence of the `big spectrum' of~$\cat T$. Given the bumpy history of the Telescope Conjecture, it could be a very difficult problem to decide in general whether the lattice of smashing $\otimes$-ideals of~$\cat T$ is spatial or not.

However, in the mathematical school of \emph{point-less topology}, a frame (or the dual notion of locale) is considered as good a topological object as topological spaces themselves; see Johnstone~\cite{Johnstone83}. In other words, one can develop topological ideas using only frames, at the cost of avoiding the points of usual point-set topology. From that perspective, if it turns out that the `big spectrum' does not always make sense as a topological space, our Theorem~\ref{thm:frame-intro} would nonetheless allow for a point-less approach to `big tt-geometry'. Future research will tell.

\medbreak

Let us briefly highlight the importance of dealing with \emph{tensor}-triangulated categories~$\cat T$, as opposed to mere triangulated categories. Already in the case of the essentially small~$\cat T^c$, the lattice of plain thick subcategories of~$\cat T^c$ can fail to be distributive. See Remark~\ref{rem:no-tensor-no-frame}. The fundamental gain of having a tensor product is that it allows us to intersect thick $\otimes$-ideals by tensoring their objects. At the end of the day the proof of distributivity as in~\eqref{eq:frame} boils down to distributivity of tensor product over coproduct, which is part of the basic properties of the tensor. The importance of the tensor will similarly be observed here, when we prove distributivity in Section~\ref{se:distribute}.

\medbreak

Beyond the lofty goal of exploring `big tt-geometry', we undertook this work with another motivation. We wanted to analyze the impact of the tensor structure of~$\cat T$ on a standard invariant of~$\cat T^c$, namely the abelian category $\MT$ of \emph{modules} over~$\cat T^c$. This Grothendieck category $\MT$ consists of additive contravariant functors from~$\cat T^c$ to abelian groups and is sometimes called its \emph{functor category}. It receives $\cat T$ via the so-called \emph{restricted-Yoneda functor}
\[
\yoneda:\cat T\too \MT
\]
defined by $\yoneda(X)=\Homcat{T}(-,X)\restr{\cat T^c}$. This construction is another example of the interaction between big~$\cat T$ and small~$\cat T^c$. Note that the functor~$\yoneda$ is usually not fully faithful outside of~$\cat T^c$ and there is no simple general description of its essential image $\yoneda(\cat T)$ in~$\MT$. For our tensor-triangulated~$\cat T$, we show in Appendix~\ref{app:MT} that $\MT$ admits a (closed) monoidal structure such that $\yoneda:\cat T\to \MT$ is monoidal. Moreover the image~$\yoneda(X)$ of every $X\in\cat T$ is \emph{flat} in~$\MT$; see Proposition~\ref{prop:tens-MT}. This appendix is of independent interest and will be used in other forthcoming work as well.

The bulk of the present article is devoted to connecting the lattice of smashing $\otimes$-ideals of~$\cat T$ to a suitable lattice of isomorphism classes of objects in~$\MT$. To understand this result, recall that smashing $\otimes$-ideals of~$\cat T$ are in one-to-one correspondence with isomorphism classes of \emph{idempotent triangles}, namely distinguished triangles in~$\cat T$
\[
\bbe\otoo{\eps} \unit \otoo{\eta} \bbf \otoo{\omega} \Sigma \bbe
\]
such that $\bbe\otimes \bbf=0$. See for instance~\cite[\S\,3]{BalmerFavi11}. Equivalently, the object $\bbf$ is a \emph{right-idempotent}, that is, $\bbf\otimes \eta$ is an isomorphism $\bbf\isoto \bbf\otimes \bbf$. The corresponding smashing subcategory is~$\cat S:=\Ker(\bbf\otimes-)=\bbe\otimes\cat T$. Restricted-Yoneda $\yoneda:\cat T\to\MT$ sends any right-idempotent~$\bbf$ to a flat right-idempotent $F=\yoneda(\bbf)$ in~$\MT$.
Smashing $\otimes$-ideals $\cat S\subseteq\cat T$ are ordered by inclusion. Similarly, isomorphism classes of right-idempotents~$[\bbf]_\simeq$ in $\cat T$, or isomorphism classes of flat right-idempotents~$[F]_\simeq$ in~$\MT$, admit a compatible partial order, that we review in Section~\ref{se:idemp}.

With these notations, our key result is the following:

\begin{Thm}
\label{thm:lattices}%
Restricted-Yoneda $\yoneda:\cat T\to \MT$ induces a lattice isomorphism:
\[
\xymatrix@C=.5em{
\left\{{{\displaystyle\textrm{smashing $\otimes$-ideals}}
 \atop{\vphantom{I^{I^I}}\displaystyle \cat S\subseteq\cat T}}\right\}
   \ar@{}[rr]|-{\cong}_-{\textrm{\cite{BalmerFavi11}}}
&&
\left\{{{\displaystyle\textrm{right-idempotents}}
 \atop{\vphantom{I^{I^I}}\displaystyle \unit \to \bbf \textrm{ in }\cat T}}\right\}_{\!\!/\simeq}
 \kern-1em  \ar[rrr]^-{\yoneda}
&&&
\left\{{{\displaystyle\textrm{flat right-idempotents}}
 \atop{\vphantom{I^{I^I}}\displaystyle\unit \to F \textrm{ in }\MT}}\right\}_{\!\!/\simeq}
}\]
whose inverse is given by $[F]_\simeq\mapsto \cat S_F:=\Ker(F\otimes\yoneda)=\SET{X\in\cat T}{F\otimes \yoneda(X)=0}$.
\end{Thm}

An immediate consequence is the fact that the lattice of
  smashing ideals of $\cat T$ is an invariant of the subcategory of
  compact objects $\cat T^c$.

The theorem will be proved at the end of Section~\ref{se:smash}; see~\ref{pt:pf-of-lattice-bijection}. In fact, we shall use another auxiliary lattice, consisting of Serre $\otimes$-ideal subcategories of~$\MT$, which will be introduced in Section~\ref{se:idemp-MT}. This connects nicely with the second author's philosophy, as presented in~\cite{Krause00,Krause05}. There, a connection between the abelian category~$\MT$ and the Telescope Conjecture was first established, without using the tensor. Indeed every smashing subcategory~$\cat S\subseteq\cat T$ is generated by a suitable ideal of \emph{morphisms} in the category~$\cat T^c$, which in turn corresponds to a Serre subcategory of~$\MT$. It follows that the lattice of non-tensor smashing subcategories is complete. However, it is known not to be distributive in general (Remark~\ref{rem:no-tensor-no-frame} again) and so it cannot be a frame, and \textsl{a fortiori} it cannot be spatial.

Although~\cite{Krause00,Krause05} are natural predecessors of the present paper, they are not formal prerequisites for our tensorial treatment, which is essentially self-contained. See further comments in Remark~\ref{rem:conclusion}. With Theorem~\ref{thm:lattices} under roof, we can prove distributivity of flat right-idempotents in the abelian category~$\MT$ and then bring it back down to the triangulated category~$\cat T$. This argument is presented in the final Section~\ref{se:distribute}.

\medbreak
\noindent
\textbf{Acknowledgments}: The first author would like to thank
Bielefeld University for its warm hospitality during a very pleasant research stay, under the auspices of the Alexander von Humboldt Foundation. We would also like to thank Billy Sanders for interesting conversations.

\goodbreak
\section{Quick review of right-idempotents}
\label{se:idemp}%
\medbreak

Let $\cat A$ be a symmetric monoidal category, with tensor $\otimes:\cat A\times \cat A\to \cat A$ and $\otimes$-unit denoted~$\unit\in\cat A$. We treat the canonical isomorphisms $X\otimes \unit\cong X \cong \unit \otimes X$ as identities, although the reader can easily restore them everywhere.

\begin{Def}\label{def:idemp}
A \emph{right-idempotent} in~$\cat A$ is a morphism $\eta:\unit\to F$ (often just denoted~$F$) such that $F\otimes \eta$ is an isomorphism $F\isoto F\otimes F$, equal to~$\eta\otimes F$.
\end{Def}

\begin{Rem}
\label{rem:idemp}%
Given a right-idempotent~$F$, the endo-functor $L=F\otimes-:\cat A\to \cat A$ together with the natural transformation $\eta\otimes-:\Id\to L$ provide a (Bousfield) localization on~$\cat A$. The essential image of~$F\otimes-$, \ie the category of local objects
\[
F\otimes \cat A :=\SET{X\in\cat A}{X\simeq F\otimes Y \textrm{ for some }Y\in\cat A}
\]
is readily seen to be $\otimes$-absorbant: $\cat A\otimes (F\otimes \cat A)\subseteq F\otimes \cat A$. We have $F\otimes \cat A=\SET{\mbox{$X\in\cat A$}}{\eta\otimes X:X\to F\otimes X\textrm{ is an isomorphism}}$. Also the functor $F\otimes-:\cat A\to F\otimes \cat A$ is left adjoint to inclusion $F\otimes\cat A \hook \cat A$. See~\cite[\S\,I.1]{GabrielZisman67} if necessary.
\end{Rem}

We can define a partial order on idempotents in several equivalent ways. The idea for $F\le F'$ is that localizing at $F'$ goes \emph{further} than localizing at~$F$.
\begin{Def}\label{def:le-idemp}
Let $\unit\oto{\eta}F$ and $\unit\oto{\eta'}F'$ be two right-idempotents (Definition~\ref{def:idemp}). We write $F\le F'$ when $F\otimes\cat A\supseteq F'\otimes\cat A$.
\end{Def}

\begin{Prop}\label{prop:idemp<}
For right-idempotents $F$ and $F'$, the following are equivalent:
\begin{enumerate}[\rm(i)]
\item $F\le F'$ in the sense of Definition~\ref{def:le-idemp}.
\item There exists an isomorphism $F'\simeq F\otimes F'$
\item $\eta\otimes F':F'\to F\otimes F'$ is an isomorphism.
\item \label{it:<-varphi} There exists a morphism $\varphi:F\to F'$ such that $\varphi\eta=\eta'$.
\end{enumerate}
If this holds, the morphism $\varphi:F\to F'$ in~\eqref{it:<-varphi} is unique: $\varphi=(\eta\otimes F')\inv(F\otimes \eta')$.
\end{Prop}

\begin{proof}
This is a lengthy exercise. See~\cite[\S\,2.12]{BoyarchenkoDrinfeld14} if necessary.
\end{proof}

\begin{Cor}\label{cor:idemp-eq}
For right-idempotents $F$ and $F'$, the following are equivalent:
\begin{enumerate}[\rm(i)]
\item $F\le F'$ and $F'\le F$.
\item The local subcategories coincide $F\otimes\cat A= F'\otimes\cat A$.
\item The objects are isomorphic $F\simeq F'$ in~$\cat A$.
\item There exists an isomorphism $\varphi:F\isoto F'$ such that $\varphi\eta=\eta'$.
\end{enumerate}
In that case, $\varphi$ is unique, $\varphi=(\eta\otimes F')\inv(F\otimes \eta')$, with inverse~$(\eta'\otimes F)\inv(F'\otimes \eta)$.
\qed
\end{Cor}

\begin{Def}
In that case we say that $F$ and $F'$ are \emph{equivalent} and we simply write $F\simeq F'$. It is clear from Definition~\ref{def:le-idemp} that $\le$ is well-defined on equivalence classes of idempotents. For instance, the initial idempotent is~$\unit$. If $\cat A$ is additive and $\otimes$ is additive in each variable then $0$ is the final idempotent.
\end{Def}

\begin{Rem}
If we know for some reason that the class of equivalence classes of right-idempotents is a set, then it is a partially ordered set under~$\le$ in view of Corollary~\ref{cor:idemp-eq}. We now indicate that any two elements always admit a supremum:
\end{Rem}

\begin{Prop}\label{prop:join}
Let $\eta_1: \unit\to F_1$ and $\eta_2: \unit \to F_2$ be right-idempotents. Then $\eta_1\otimes\eta_2:\unit\to F_1\otimes F_2$ is a right-idempotent which is the join of~$F_1$ and~$F_2$:
\[
 F_1\otimes F_2 = F_1 \join F_2 \equalby{def.} \sup\{F_1,F_2\}\,.
\]
\end{Prop}

\begin{proof}
Straightforward consequence of Definition~\ref{def:le-idemp}.
\end{proof}

\goodbreak
\begin{center}
*\ *\ *
\end{center}
\goodbreak

\begin{Hyp}\label{hyp:ab-tens}
We further assume that $\cat A$ is \emph{abelian}. We assume $\otimes$ to be additive in each variable but we do \emph{not} assume that $\otimes$ is exact on either side. So everything we say here holds for the opposite category~$\cat A\op$, \ie for \emph{left-idempotents}.
\end{Hyp}

\begin{Def}
An object $F\in\cat A$ is called \emph{flat} if $F\otimes-:\cat A\to \cat A$ is exact.
\end{Def}

\begin{Rem}\label{rem:idemp-ab}
Let $\eta:\unit\to F$ be such that $F\otimes\eta$ is an isomorphism $F\isoto F\otimes F$. Then $\eta\otimes F$ is also an isomorphism but let us not assume that it agrees with $F\otimes \eta$.
Consider the full subcategory $\Ker(F\otimes-):=\SET{X\in\cat A}{F\otimes X=0}$ of~$\cat A$.
\begin{enumerate}[\rm(a)]
\item\label{it:perp}
For every $X\in\Ker(F\otimes-)$ and $Y\in F\otimes\cat A$, we have $\Homcat{A}(X,Y)=0$. Indeed, let $\alpha:X\to Y$ with $F\otimes X=0$ and $Y$ such that $\eta\otimes Y:Y\isoto F\otimes Y$ (Remark~\ref{rem:idemp}). Composing $\alpha$ with this isomorphism gives $(\eta\otimes Y)\circ\alpha=\eta\otimes \alpha=(F\otimes\alpha)\circ(\eta\otimes X)$ which factors via $F\otimes X=0$. This shows that $\alpha=0$.
\smallbreak
\item\label{it:eta-F}
Suppose moreover that $F$ is flat. Then automatically $F\otimes\eta=\eta\otimes F$. Indeed, $\coker(\eta)$ belongs to~$\Ker(F\otimes-)$ by flatness of~$F$; as $(F\otimes \eta)\eta=\eta\otimes \eta=(\eta\otimes F)\eta$, we see that $F\otimes \eta-\eta\otimes F$ factors via a morphism $\coker(\eta)\to F\otimes F$ which must be the zero morphism by~\eqref{it:perp}, since $\coker(\eta)\in\Ker(F\otimes-)$ and $F\otimes F\in F\otimes \cat A$.
\end{enumerate}
\end{Rem}

\begin{Def}\label{def:flat-idemp}
A \emph{flat right-idempotent} $F$ is a morphism $\eta:\unit\to F$ with $F$ flat and $F\otimes \eta:F\isoto F\otimes F$. By Remark~\ref{rem:idemp-ab}\,\eqref{it:eta-F}, this is also an idempotent in the sense of Definition~\ref{def:idemp}. The partial order remains as before (Definition~\ref{def:le-idemp}).
\end{Def}

We can now add to the list of Proposition~\ref{prop:idemp<} an abelian condition:

\begin{Prop}\label{prop:idemp<-abel}
For flat right-idempotents $F$, $F'$, the following are equivalent:
\begin{enumerate}[\rm(i)]
\item $F\le F'$ in the sense of Definition~\ref{def:le-idemp} (see also Proposition~\ref{prop:idemp<}).
\item $\Ker(F\otimes-)\subseteq\Ker(F'\otimes-)$.
\end{enumerate}
\end{Prop}

\begin{proof}
(i) clearly implies~(ii). Conversely, if (ii) holds then $\ker(\eta),\coker(\eta)\in\Ker(F\otimes-)\subseteq\Ker(F'\otimes -)$ shows that $F'\otimes \eta$ is an isomorphism, hence~(i).
\end{proof}

\begin{Cor}\label{cor:idemp-eq-ab}
For flat right-idempotents $F$, $F'$, the following are equivalent:
\begin{enumerate}[\rm(i)]
\item $F$ and $F'$ are equivalent: $F\simeq F'$ (see also Corollary~\ref{cor:idemp-eq}).
\item $\Ker(F\otimes-) = \Ker(F'\otimes-)$.
\qed
\end{enumerate}
\end{Cor}

\begin{Cor}
\label{cor:set-coh}%
If $\cat A$ is a Grothendieck category and $\otimes$
  preserves coproducts, then there is only a set of isomorphism classes of flat right-idempotents.
\end{Cor}

\begin{proof}
If $F$ is a flat right-idempotent then $\Ker(F\otimes-)$ is a localizing Serre subcategory of $\cat A$ and a Grothendieck category only has a set of such subcategories.
\end{proof}

\goodbreak
\section{From right-idempotents to Serre subcategories}
\label{se:idemp-MT}%
\medbreak

Let $\cat T$ be a rigidly-compactly generated tensor
  triangulated category and
\[
\cat A:=\MT
\]
the Grothendieck category of modules over~$\cat T^c$. For details
about~$\cat T$ and a general introduction to the module
category~$\MT$, the reader is invited to consult
Appendix~\ref{app:MT}.

In this section, we associate a Serre subcategory to any flat
right-idempotent $\unit\oto{\eta}F$ (Definition~\ref{def:flat-idemp})
in $\cat A$. The critical fact is the following:
\begin{Lem}[Idempotence Lemma]
\label{lem:I^2}%
Let $\eta:\unit\to F$ be a flat right-idempotent in~$\cat A$ and let $I=\ker(\eta)\ointo{\iota} \unit$ its kernel. Then the morphism $\iota\otimes I:I\otimes I\to I$ is an epimorphism, equal to~$I\otimes\iota$.
\end{Lem}

\begin{Rem}
\label{rem:I^2}%
If we think of $\unit\in\cat A$ as a ring in the trivial way $\unit\otimes\unit\cong\unit$, a subobject~$I\into \unit$ is simply an ideal. If we define the square ideal $I^2\into \unit$ as the image $\img(I\otimes I\to \unit)$ inside $\unit$, then Lemma~\ref{lem:I^2} reads $I^2=I$. Hence the name.
\end{Rem}

\begin{proof}[Proof of Lemma~\ref{lem:I^2}:]
Since $\iota\circ (I\otimes \iota)=\iota\otimes \iota=\iota\circ(\iota\otimes I)$ and $\iota$ is a monomorphism, we have $I\otimes \iota=\iota\otimes I$.
Let us first establish some general facts about the abelian category $\cat A=\MT$. Since $\cat A$ has enough flats (see Propositions~\ref{prop:enough} and~\ref{prop:tens-MT}), we can consider the derived functors $\Tor_*$ of~$\otimes$. The auto-equivalence $\Sigma:\cat A\isoto\cat A$ from~\ref{pt:sigma} extends suspension~$\Sigma:\cat T\isoto \cat T$ and is compatible with the tensor in that $M\otimes\Sigma N\cong \Sigma(M\otimes N)\cong (\Sigma M)\otimes N$. On~$\Tor_*$ it yields: $\Tor_*(M,\Sigma N)\cong \Sigma\Tor_*(M,N)\cong \Tor_*(\Sigma M,N)$. Every object $M\in\cat A$ fits in an exact sequence
\begin{equation}
\label{eq:M-3}%
0\to M\to P_0\to P_1\to P_2\to \Sigma M\to 0
\end{equation}
with $P_0,P_1,P_2$ flat, by Propositions~\ref{prop:enough} and~\ref{prop:tens-MT}.
This yields natural isomorphisms
\begin{equation}\label{eq:Tor-Sigma}
\Tor_{i}(M,N) \cong \Tor_{i+3}(\Sigma M,N) 
\end{equation}
for all $i\ge 1$, all $M$ and $N$ in~$\cat A$. Moreover, if $L=\img(P_0\to P_1)$ in~\eqref{eq:M-3}, we have two exact sequences $0\to M\to P_0 \to L\to 0$ and $0\to L \to P_1\to P_2\to \Sigma M\to 0$, from which we deduce respectively a natural monomorphism $\Tor_1(L,N)\into M\otimes N$ and an isomorphism $\Tor_3(\Sigma M,N)\cong \Tor_1(L,N)$ for all~$N\in\cat A$. Combining them, we have a natural monomorphism
\begin{equation}\label{eq:Tor-3}
\Tor_{3}(\Sigma M,N) \into M\otimes N
\end{equation}
for all~$M,N\in\cat A$. This concludes the generalities about~$\cat A$.

Consider now the image~$C$ and the cokernel~$K$ of our morphism~$\eta$:
\[
\xymatrix@R=1em{
& \unit \ar[rr]^-{\eta} \ar@{->>}[rd]_-{\beta}
&& F \ar@{->>}[rd]^-{\delta}
\\
I \ar@{ >->}[ru]^-{\iota}
&& C \ar@{ >->}[ru]_-{\gamma}
&& K.
 }
\]
The strategy is to hit everything in sight with~$\otimes$. First, tensoring this diagram with the flat object~$F$ and using the assumption that $F\otimes\eta$ is an isomorphism, we get
\begin{equation}
\label{eq:F-IKC}%
F\otimes I=0,
\qquad
F\otimes K=0
\qquadtext{and}
F\otimes \beta:F \isoto F\otimes C\,.
\end{equation}
The relation $F\otimes K=0$ and the epimorphism $\delta \otimes K: F\otimes K\onto K\otimes K$ force $K\otimes K=0$. Combining with~\eqref{eq:Tor-3}, we see that $\Tor_3(K,K)=0$ and using the short exact sequence $C\into F\onto K$ and flatness of~$F$ twice, we deduce
\begin{equation}
\label{eq:Tor-CC}%
0=\Tor_3(K,K)\cong \Tor_2(C,K)\cong \Tor_1(C,C)\,.
\end{equation}
From $F\otimes I=0$ in~\eqref{eq:F-IKC} and the epimorphism $\delta\otimes I:\, F\otimes I \onto K\otimes I$, we see that $K \otimes I=0$. The long $\Tor_*(K,-)$ exact sequence associated to $I\into \unit \onto C$ then gives
\begin{equation}
\label{eq:K-KC}%
K\otimes\beta:\ K\isoto K\otimes C
\qquadtext{and}
\Tor_{1}(K,C) = 0.
\end{equation}
The long $\Tor_*(-,C)$ exact sequence associated to $C\into F \onto K$ therefore gives the exact sequence at the top of the following commutative diagram:
\[
\xymatrix{
0 \ar[r]
& \ {\underbrace{\Tor_{1}(K,C)}}_{=0} \ar[r]
& C\otimes C \ar[r]^-{\gamma\otimes C}
& F\otimes C \ar[r]^-{\delta\otimes C}
& K\otimes C \ar[r]^-{}
& 0
\\
& 0 \ar[r]
& C \ar[r]^-{\gamma} \ar[u]^-{C \otimes \beta}
& F \ar[r]^-{\delta} \ar[u]_-{\simeq}^-{F \otimes \beta}
& K \ar[r]^-{} \ar[u]_-{\simeq}^-{K \otimes \beta}
& 0
}
\]
The second row is just one of our initial short exact sequences and the vertical isomorphisms were obtained in~\eqref{eq:F-IKC} and~\eqref{eq:K-KC}. We have therefore an isomorphism $C\otimes \beta:\ C \isoto C\otimes C$. Using this isomorphism $C\otimes \beta$ in the long $\Tor_*(C,-)$ exact sequence associated to $I\into \unit \oonto{\beta} C$ gives $C\otimes I\simeq \Tor_{1}(C,C)$ which is zero by~\eqref{eq:Tor-CC}. Tensoring $I\ointo{\iota}\unit \onto C$ with~$I$ now gives the surjectivity of $\iota\otimes I:\ I\otimes I\onto I$.
\end{proof}

Recall that $\widehat{(-)}:\cat T\to \MT$ denotes restricted-Yoneda and $\mT$ the finitely presented objects in $\MT$. See~\ref{pt:remind-MT}. Here is a general observation:

\begin{Lem}
\label{lem:I-colim}%
Let $\iota:I\into \unit$ in~$\MT$, or any subobject of some~$\hat z$ for $z\in\cat T^c$. Then $I$ is the filtered colimit (union) of its finitely presented subobjects.
\end{Lem}

\begin{proof}
This holds because $\MT$ is locally coherent (see~\ref{pt:coh}): Every object is the filtered colimit of its finitely generated subobjects, which are images of finitely presented ones. Let $m\in\mT$ and $\mu:m\to I$. The image of the composite $\iota\mu:m\to \unit$ is finitely presented because both~$m$ and~$\unit$ are. Since $\iota$ is a monomorphism, $\img(\iota\mu)\cong \img(\mu)$ and it follows that $\img(\mu)\subseteq I$ is finitely presented.
\end{proof}

We can now associate to a flat right-idempotent an interesting Serre subcategory.

\begin{Thm}
\label{thm:F-to-Serre}%
Let $\unit\oto{\eta}F$ be a flat right-idempotent in~$\cat A=\MT$. Consider the localizing Serre $\otimes$-ideal of~$\cat A$ (see~\ref{pt:B-loc})
\begin{equation}
\label{eq:BF}%
\cat B_F:=\Ker(F\otimes-)=\SET{M\in\cat A}{F\otimes M=0}
\end{equation}
and the corresponding ideal of morphisms~$\calJ_F$ in~$\cat T^c$, defined by
\begin{equation}
\label{eq:JF}%
\calJ_F:=\SET{f:x\to y\textrm{ in }\cat T^c}{\img(\hat f)\in\cat B_F}\,.
\end{equation}
Then $\cat B_F$ is the Serre $\otimes$-ideal of~$\cat A$ generated by the object~$I=\ker(\eta)$. Moreover:
\begin{enumerate}[\rm(a)]
\item
\label{it:BF-a}%
The subcategory $\cat B_F$ is generated as localizing subcategory of~$\cat A$ by its finitely presented part $\cat B_F\fp=\SET{m\in\cat A\fp}{F\otimes m=0}$, or in formula: $\cat B_F=(\cat B_F\fp)^\to$.
\smallbreak
\item
\label{it:BF-b}%
The ideal of morphisms~$\calJ_F$ is idempotent $\calJ_F=\calJ_F^2$, \ie every morphism $f\in\calJ_F(x,y)$ can be written as $f=h\,g$ with $g\in\calJ_F(x,z)$, $h\in\calJ_F(z,y)$ and $z\in\cat T^c$.
\end{enumerate}
\end{Thm}

\begin{proof}
As $F$ is flat, it is clear that $\cat B_F=\Ker(F\otimes-)$ is a localizing Serre $\otimes$-ideal; see~\ref{pt:B-loc}. Let $M\in\cat B_F$. Let $\alpha:M\into Y$ be any monomorphism in~$\cat A$ with $Y$ flat. This always exists, for instance with $Y=\hat Z$ for $Z\in\cat T$, by Proposition~\ref{prop:enough}. Consider the following commutative diagram whose second row is exact by flatness of~$Y$:
\[
\xymatrix@R=2em{
&& M \ar[r]^-{\eta\otimes M} \ar@{ >->}[d]^-{\alpha} \ar@{-->}[ld]
& F\otimes M =0 \kern-2em \ar[d]^-{F\otimes \alpha}
\\
0 \ar[r]
& I\otimes Y \ar[r]^-{\iota\otimes Y}
& Y \ar[r]^-{\eta\otimes Y}
& F \otimes Y
}
\]
The assumption that $F\otimes M=0$ guarantees the factorization of $\alpha:\,M\into Y$ via $I\otimes Y$. Since $\alpha$ is a monomorphism, so is the (dotted) arrow $M\to I\otimes Y$, which shows that $M$ belongs to the Serre $\otimes$-ideal generated by~$I$. We have shown
\begin{equation}
\label{eq:BF-I}%
\cat B_F=\SET{M\in\cat A}{\textrm{every }\alpha:M\into Y\textrm{ with $Y$ flat factors via }I\otimes Y\into Y}\,.
\end{equation}
Since $F\otimes I=0$, see~\eqref{eq:F-IKC}, we have $I\in\cat B_F$ and we have therefore shown that $\cat B_F$ is exactly the (localizing) Serre $\otimes$-ideal generated by~$I$. As the object~$I$ itself is generated by its finitely presented subobjects by Lemma~\ref{lem:I-colim}, we see that $\cat B_F$ is generated by the finitely presented subobjects of~$I$, hence we have shown~\eqref{it:BF-a}.

Let us now prove~\eqref{it:BF-b}. Let $f:x\to y$ in~$\cat T^c$ belong to~$\calJ_F$, that is, such that the object $\img(\hat f)$ belongs to~$\cat B_F$. By~\eqref{eq:BF-I} (applied to $M=\img(\hat f)\into \hat y$), we have the following description of the ideal of morphisms~$\calJ_F$:
\begin{equation}
\label{eq:JF-I}%
\calJ_F=\SET{f:x\to y\textrm{ in }\cat T^c}{\hat f:\hat x\to \hat y\textrm{ factors via }\iota\otimes\hat y:\,I\otimes \hat y\into \hat y}\,.
\end{equation}
Hence we have a factorization $\hat x \oto{\xi} I\otimes \hat y\,\ointo{\iota\otimes 1}\,\hat y$ of our~$\hat f$. By Lemma~\ref{lem:I-colim}, the object~$I=\ker(\eta)$ is the filtered colimit of its finitely presented subobjects\,: $I=\colim_{m\subseteq I\textrm{, f.p.}}m$. Combining this with the epimorphism $I\otimes \iota:\,I\otimes I\onto I$ of the Idempotence Lemma~\ref{lem:I^2} and commuting the colimit and the tensor, $I\otimes I = \colim_{m_1,m_2\subseteq I\textrm{, f.p.}} \ m_1\otimes m_2 \ \onto \ I$,
we see that $I$ is the filtered colimit (union) of the images of the~$m_1\otimes m_2\to I$ over~$m_1,m_2\subseteq I$ finitely presented. Tensoring further with~$\hat y$, we have a similar epimorphism
\[
I\otimes I \otimes \hat y= \colim_{m_1,m_2\subseteq I\textrm{, f.p.}} m_1\otimes m_2  \otimes \hat y \quad \onto \ I \otimes \hat y.
\]
Returning to our morphism $\hat f:\hat x \oto{\xi} I\otimes\hat y \oto{\iota\otimes 1}\hat y$, we build a commutative diagram
\[
\xymatrix@C=3em{
&&&&&
\\
\hat x \ar[rr]^-{\xi} \ar@{-->}_-{\exists}[rrdd]^-{(1)} \ar@{-->}[rdd]_-{\exists}^-{(2)} \ar@/_2em/@{-->}[rddd]_-{\exists}^-{(3)}
 \ar `l[ddd] `[dddd] `[rrrddd]^-{\hat g} [rrrddd] \ar `u[rrru] `[rrr]_-{\hat f} [rrr]
&& I \otimes \hat y \ar@{ >->}[r]^-{\iota\otimes \hat y}
& \hat y\
\\
&& I \otimes I \otimes \hat y \ar@{=}[d] \ar@{->>}[u]^-{I\otimes\iota \otimes \hat y}
\\
& m_1\otimes m_2 \otimes \hat y \ar[r]^-{\textrm{can}}
& \colim\limits_{m_1,m_2\subseteq I\textrm{, f.p.}} m_1\otimes m_2 \otimes \hat y
& m_1 \otimes \hat y \ar@{ >->}[uu]^-{\iota_1\otimes \hat y}
\\
& \hat x_1 \otimes m_2 \otimes \hat y \ar@{->>}[u]_-{\pi_1\otimes m_2\otimes \hat y} \ar[rr]_-{\hat x_1\otimes \iota_2\otimes \hat y}
&& \hat x_1 \otimes \hat y \ \ar@{->>}[u]^-{\pi_1\otimes \hat y}
 \ar `r[uuu] `[uuu]_-{\hat h} [uuu]
\\
&&&&&}
\]
Starting at the top, the morphism $\xi:\hat x\to I\otimes \hat y$ can be lifted~(1) along the epimorphism $I\otimes I\otimes \hat y \onto I\otimes \hat y$ (since $\hat x$ is projective); it lifts further~(2) to some $m_1\otimes m_2\otimes \hat y$ in the filtered colimit (since $\hat x$ is finitely presented); finally, if $\pi_1:\,\hat x_1\onto m_1$ is any epimorphism with $x_1\in\cat T^c$, then we can lift further~(3) to $\hat x_1\otimes m_2 \otimes\hat y$ (using $\hat x$ projective again). The anti-diagonal composite $\hat x_1\otimes m_2\otimes\hat y \too \hat y$ is $(\iota_1\pi_1)\otimes\iota_2\otimes\hat y$, where $\iota_i:m_i\to \unit$ is simply the composite $m_i\into I\ointo{\iota}\unit$, for $i=1,2$. This also decomposes as in the above picture as $(\iota_1\otimes\hat y)\,(\pi_1\otimes\hat y)\,(\hat x_1\otimes \iota_2\otimes\hat y)$. Since the Yoneda $\otimes$-functor $\cat T^c\to \cat A\fp$ is fully faithful, we have shown that our morphism $f\in\calJ(x,y)$ factors as a morphism~$g:x\to x_1\otimes y$ followed by a morphism~$h:x_1\otimes y\to y$; the corresponding $\hat g$ and $\hat h$ factor via objects of~$\cat B_F$ since $m_1,m_2\in\cat B_F$, as subobjects of~$I$. Hence $g\in\calJ(x,x_1\otimes y)$ and $h\in\calJ(x_1\otimes y,y)$ as wanted.
\end{proof}

\begin{Rem}
\label{rem:BF-inj}%
It is clear that the subcategory $\cat B_F$ only depends on the isomorphism class~$[F]_{\simeq}$ of the flat right-idempotent~$F$. Furthermore, if $F\le F'$ then $\cat B_{F}\subseteq\cat B_{F'}$. See Proposition~\ref{prop:idemp<-abel}. Thus the assignment $[F]_{\simeq}\mapsto \cat B_F$ is order-preserving and injective on isomorphism classes of idempotents by Corollary~\ref{cor:idemp-eq-ab}.
\end{Rem}

\begin{Rem}
\label{rem:BJ}%
In Theorem~\ref{thm:F-to-Serre}, we associate three structures to a given flat right-idempotent~$F$ in~$\cat A=\MT$: First, the big Serre subcategory~$\cat B:=\Ker(F\otimes-)$ of~$\cat A$, second the small Serre subcategory $\cat B\fp$ of~$\cat A\fp$ and third the ideal $\calJ$ of those morphisms in~$\cat T^c$ whose images belong to~$\cat B\fp$. There is an obvious redundancy in this data, as every single one determines the other two. For instance, given $\calJ$, we have $\cat B\fp=\SET{\img(\hat f)}{f\in\calJ}$ and $\cat B=(\cat B\fp)^\to$. See~\ref{pt:B-loc}. One can therefore try to express all properties of the trio in terms of just one of its members.

Being $\otimes$-ideal has the obvious equivalent formulations for $\cat B$, $\cat B\fp$ and $\calJ$. The crucial idempotence property is best expressed for $\calJ$, as $\calJ=\calJ^2$. In terms of~$\cat B$, it can be shown to be equivalent to $\cat B\fp$ being \emph{perfect} in the terminology of~\cite[\S\,8]{Krause05}, which means that the right adjoint $\cat A/\cat B\hook \cat A$ to the Gabriel quotient $\cat A\onto \cat A/\cat B$ is an exact functor. The `Serre' property on the other hand is clearly best expressed for~$\cat B$ or $\cat B\fp$. If one wants to translate it in terms of~$\calJ$, one needs to characterize the fact that the associated subcategory $\SET{\img(\hat f)}{f\in\calJ}$ of~$\cat A\fp$ is Serre. This condition could be rephrased by saying that $\calJ$ is \emph{cohomological}, which admits various equivalent formulations, see~\cite{Krause05}.

We leave these variations to the reader. Our real contribution here is the new connection between idempotents in~$\MT$ and smashing ideals of~$\cat T$. Hence we use both $\cat B$ and $\calJ$ simultaneously, as mere auxiliaries. See also Remark~\ref{rem:conclusion}.
\end{Rem}

\goodbreak
\section{From Serre subcategories to smashing subcategories}
\label{se:smash}%
\medbreak

We check in complete generality that the conclusions about the Serre $\otimes$-ideal of~$\cat A=\MT$ obtained in Theorem~\ref{thm:F-to-Serre}\,\eqref{it:BF-a} and~\eqref{it:BF-b} are sufficient to construct a smashing $\otimes$-ideal of~$\cat T$. Terminology is reviewed in~\ref{pt:B-loc} if necessary.
\begin{Thm}
\label{thm:B-to-smash}%
Let $\cat B\subseteq\cat A=\MT$ be a localizing Serre $\otimes$-ideal and consider the corresponding ideal $\calJ=\SET{f:x\to y\textrm{ in }\cat T^c}{\img(\hat f)\in\cat B}$ of morphisms in~$\cat T^c$. Assume that $\cat B$ and $\calJ$ satisfy the following properties (see Remark~\ref{rem:BJ}):
\begin{enumerate}[\rm(i)]
\item
\label{it:hyp-B-i}%
We have $\cat B=(\cat B\fp)^\to$, \ie the subcategory $\cat B$ is generated by~$\cat B\fp=\cat A\fp\cap\cat B$.
\smallbreak
\item
\label{it:hyp-B-ii}%
The ideal $\calJ=\calJ^2$ is \emph{idempotent}, \ie each $f\in\calJ$ equals $h\,g$ for some~$h,g\in\calJ$.
\end{enumerate}
Define the following subcategory of~$\cat T$:
\begin{equation}
\label{eq:S}%
\cat S:=\yoneda\inv(\cat B)=\SET{X\in\cat T}{\hat X\in\cat B}\,.
\end{equation}
Then $\cat S$ is a smashing $\otimes$-ideal of~$\cat T$. Moreover, for every morphism $f:x\to y$ in~$\cat T^c$:
\begin{equation}
\label{eq:J-f}%
{\displaystyle\textrm{$f$ belongs to~$\calJ$, \ie $\img(\hat f)\in\cat B$, if and only if $f$ factors via an object of~$\cat S$.}}
\end{equation}
\end{Thm}

\begin{proof}
Clearly $\cat S$ is a $\otimes$-ideal localizing subcategory of~$\cat T$, since $\yoneda:\cat T\to \cat A$ is a homological $\otimes$-functor commuting with coproducts and $\cat B$ is a localizing Serre $\otimes$-ideal of~$\cat A$. (See~\ref{rem:sigma}.) To prove that~$\cat S$ is smashing, we need some preparation.

We present an argument \`a la Neeman. It will use countable sequences
\begin{equation}
\label{eq:x_n's}%
x_0 \otoo{f_1}
x_1 \otoo{f_2}
\cdots
\too x_n \otoo{f_n}
x_{n+1} \otoo{f_{n+1}}
\cdots
\end{equation}
of morphisms $f_n\in\calJ$ in the ideal~$\calJ$ associated to~$\cat B$. We shall give a name to the collection of all homotopy colimits in~$\cat T$ (see~\cite[\S\,1.6]{Neeman01}) obtained in this way:
\begin{equation}
\label{eq:G}%
\cat G:=\SET{\hocolim_{n\in \bbN}x_n}{\cdots x_n\oto{f_n} x_{n+1}\cdots \textrm{ as in~{\eqref{eq:x_n's}}, with all $f_n$ in }\calJ}\,.
\end{equation}
Since $\cat T^c$ is essentially small, this collection $\cat G\subset\cat T$ is essentially small as well.

\begin{Lem}
\label{lem:G-S}%
With the above notation, we have $\cat G\subseteq \cat S$ and $\cat S$ is the localizing subcategory of~$\cat T$ generated by~$\cat G$.
\end{Lem}

\begin{proof}
The hocolim of a sequence as in~\eqref{eq:x_n's} fits in a distinguished triangle in~$\cat T$
\[
\coprod_{n\in \bbN} x_n \ \otoo{1-\tau\,\,} \
\coprod_{n\in \bbN} x_n \to
\hocolim_{n\in\bbN} x_n
\to \Sigma \coprod_{n\in \bbN} x_n
\]
where $\tau$ is defined on each summand by the `shift' $x_n\oto{f_n}x_{n+1}\hook\coprod_\ell x_\ell$. Under restricted-Yoneda~$\yoneda:\cat T\to \MT=\cat A$ this yields a long exact sequence, which in turns maps to a long exact sequence in the Gabriel quotient $\cat A/\cat B$. However, in~$\cat A/\cat B$ the morphism $1-\tau$ is the identity since all components $f_n$ of~$\tau$ have their image in~$\cat B$; hence the `third' object in that sequence is zero in~$\cat A/\cat B$, which means that $\yoneda(\hocolim_nx_n)\in\cat B$ as wanted. Hence $\cat G\subseteq\cat S$.

Fix an object $Y$ in~$\cat S=\yoneda\inv(\cat B)$. Every morphism $f:x\to Y$ with $x\in\cat T^c$ maps under restricted-Yoneda to a morphism $\hat f$ whose target belongs to~$\cat B$ and $\cat B=(\cat B\fp)^\to$ by Hypothesis~\eqref{it:hyp-B-i}. See~\ref{pt:B-loc}. Therefore $\hat f:\hat x\to \hat Y$ factors via a morphism $\hat x\to m$ with $m\in \cat B\fp$. Replacing $m$ by the image of~$\hat x$ if necessary (which is still finitely presented since both~$\hat x$ and~$m$ are), we can assume that $\hat x\onto m$ is onto. Choosing a presentation $\hat z\to \hat x \onto m$ with $z\in\cat T^c$ we can complete $z\to x$ into a distinguished triangle $z\to x \oto{g} y \to \Sigma z$ in~$\cat T^c$ and obtain the following (plain) picture in~$\cat A$:
\[
\xymatrix@R=1em{
\hat z \ar[r] \ar[rdd]_-{=0}
& \hat x \ar[rr]^-{\hat g} \ar@{->>}[rd] \ar[dd]_-{\hat f}
&& \hat y \ar@{-->}@/^1.5em/[lldd]^-{\hat h} \ar[r]
& \Sigma\hat z
\\
&& m \ar@{ >->}[ru] \ar[ld]
\\
& \hat Y
}
\]
The composite $\hat z\to \hat x \to \hat Y$ is zero because it factors via~$m$ and the composite $\hat z\to \hat x\to m$ is already zero. By Yoneda's Lemma, the composite $z\to x \oto{f} Y$ is also zero (see~\ref{eq:yoneda-hom}) and therefore there exists $h:y\to Y$ in~$\cat T$ such that $f=h\,g$. Note that $\hat g$ factors via $m\in\cat B\fp$ so $g\in\calJ(x,y)$.

So we have proved that every morphism $f:x_1\to Y$ with $x_1\in\cat T^c$ factors as the composite of a morphism in~$\calJ(x_1,x_2)$ followed by a morphism $x_2\to Y$ with~$x_2\in\cat T^c$, with which we can repeat the argument. We obtain therefore by induction a sequence of morphisms in~$\calJ$ and a factorization of the original~$f$ via $\hocolim_{n}x_n$:
\[
\xymatrix{
x_1 \ar[r]^-{\in\calJ} \ar[d]^-{f}
& x_{2} \ar[r]^-{\in\calJ} \ar[ld]^-{}
& x_{3} \ar[r]^-{\in\calJ} \ar[lld]^-{}
& \ar@{}[r]|-{\cdots }
& \hocolim_{n}x_{n}=:x_{\infty} \ar[lllld]
\\
Y
}
\]
We have $x_{\infty}\in\cat G$ by definition, see~\eqref{eq:G}. As $x_1\in\cat T^c$ was arbitrary, this proves the implications $\Homcat{T}(\cat G,Y)=0 \implies \Homcat{T}(\cat T^c,Y)=0\implies Y=0$. In the terminology of~\cite{Neeman01}, we have shown that $\cat G\subseteq \cat S$ is a \emph{generating set}. By Neeman~\cite[Prop.\,8.4.1]{Neeman01}, $\cat S$ is generated by~$\cat G$ as a localizing subcategory.
\end{proof}

We can now finish the proof of Theorem~\ref{thm:B-to-smash}\,: In Lemma~\ref{lem:G-S}, we proved that $\cat S$ is the localizing subcategory of~$\cat T$ generated by a \emph{set} of objects, namely any skeleton of~$\cat G$. It follows from~\cite[\S\,8.4]{Neeman01} that $\cat S$ is moreover \emph{strictly} localizing meaning that the inclusion $\cat S\into \cat T$ admits a right adjoint, \ie we obtain a Bousfield localization:
\[
\xymatrix@R=1.7em{
\cat S \ar@<-.3em>@{ >->}[d]_-{\incl\ }
\\
\cat T \ar@<-.3em>@{->>}[d]_-{Q\ } \ar@{->>}@<-.35em>[u]_-{\ \Gamma}
\\
\cat T/\cat S \ar@<-.35em>@{ >->}[u]_-{\ R}
}
\]
In particular, we have a distinguished triangle
\begin{equation}
\label{eq:B-ef}%
\bbe \oto{\eps} \unit \oto{\eta} \bbf \to \Sigma \bbe
\end{equation}
where $\bbe:=\Gamma(\unit)$ belongs to~$\cat S$ and $\bbf:=RQ(\unit)\in\Ker(\Gamma)=\cat S^\perp$ is `local'. Proving that $\cat S$ is smashing amounts to showing that $\bbe \otimes \bbf=0$, \ie that the above triangle is idempotent. Equivalently, we have to show that the localization functor $RQ:\cat T\to \cat T$ is given by ~$RQ(\unit)\otimes-$, which we denoted $\bbf\otimes-$. It is a general fact that there is a natural isomorphism $\bbf \otimes x \cong RQ(x)$ for all compact $x\in\cat T^c$ because these~$x$ are rigid. We recall the argument for completeness; it suffices to test under Yoneda:
\begin{align*}
\Hom & (?,\bbf\otimes x) \cong \Hom(?\otimes x^\vee, RQ(\unit)) \cong \Hom(Q(?\otimes x^\vee), Q(\unit))
\\
& \cong \Hom(Q(?)\otimes Q(x)^\vee, Q(\unit)) \cong \Hom(Q(?), Q(\unit)\otimes Q(x)) \cong \Hom(?, RQ(x))\,.
\end{align*}
In particular, if $f:x\to y$ is a morphism in~$\cat T^c$, we have $Q(f)=0$ if and only if $\bbf\otimes f=0:\bbf\otimes x\to \bbf \otimes y$. The condition $Q(f)=0$ is also equivalent to $f$ factoring via an object of~$\cat S$. It is now a good point to establish property~\eqref{eq:J-f} in the statement.

Let $f:x\to y$ in~$\cat T^c$. Of course, if $f$ factors via some $Y\in\cat S$ then $\hat f$ factors via $\hat Y\in\cat B$. Let us assume the converse, that is, assume that $\img(\hat f)\in\cat B$, or in other words that $f\in\calJ$. This is the place where we use Hypothesis~\eqref{it:hyp-B-ii}, which implies that $f=h\,f_1$ for some $f_1,h\in\calJ$. Repeating and using induction, we produce a countable sequence of morphisms in~$\calJ$ as follows:
\[
\xymatrix@C=4em{
x=x_1 \ar[r]^-{f_1\in\calJ} \ar[d]|-{\ f\in\calJ\ }
& x_{2} \ar[r]^-{\ f_2 \in\calJ\ } \ar[ld]|-{\ \in\calJ\ }
& x_{3} \ar[r]^-{\ f_3\in\calJ\ } \ar[lld]|-{\ \in\calJ\ }
& \ar@{}[r]|-{\cdots }
& \hocolim_{n}x_{n}=:x_{\infty} \ar[lllld]
\\
y \ar@{}[rrru]|-{\cdots }
}
\]
Again, $x_{\infty}\in\cat G$ by~\eqref{eq:G} and therefore $x_{\infty}\in\cat S$ by Lemma~\ref{lem:G-S}. This proves that $f$ factors via an object of~$\cat S$, as claimed in~\eqref{eq:J-f}.

Putting everything together, we have shown that if $f\in\calJ$ then $f$ factors via an object of~$\cat S$, hence $RQ(f)=0$, hence $\bbf\otimes f=0$.

It follows that for any countable sequence of morphisms in~$\calJ$ as in~\eqref{eq:x_n's} we have $\bbf\otimes f_n=0$ for all~$n$ and therefore $\bbf\otimes \hocolim_{n}x_n$ is the hocolim of a sequence of zero morphisms, hence is zero. In other words, $\bbf\otimes\cat G=\{0\}$ in the notation of~\eqref{eq:G}. It follows that $\bbf\otimes \cat S=0$ since $\cat S$ is generated by~$\cat G$ by Lemma~\ref{lem:G-S}. In particular $\bbe\otimes\bbf=0$ and~\eqref{eq:B-ef} is an idempotent triangle. See~\ref{pt:remind-big-tt}. So, $\bbe\otimes \cat T\subseteq\cat S\subseteq \Ker(\bbf\otimes-)=\bbe\otimes\cat T$ which shows that $\cat S=\bbe\otimes \cat T=\Ker(\bbf\otimes-)$ is indeed smashing.
\end{proof}

\begin{Cor}
\label{cor:B-to-smash-inj}%
Let $\cat B$ and $\cat B'$ be two localizing Serre $\otimes$-ideals of~$\MT$ satisfying Hypotheses~\eqref{it:hyp-B-i} and~\eqref{it:hyp-B-ii} of Theorem~\ref{thm:B-to-smash}. Suppose that they induce the same smashing subcategories $\yoneda\inv(\cat B)=\yoneda\inv(\cat B')$ of~$\cat T$. Then $\cat B=\cat B'$.
\end{Cor}

\begin{proof}
Because of~\eqref{eq:J-f}, $\cat B$ and $\cat B'$ induce the same ideal of morphisms $\calJ=\calJ\,'$ in~$\cat T^c$. Hence $\cat B\fp=\SET{\img(\hat f)}{f\in\calJ}=\SET{\img(\hat f)}{f\in\calJ\,'}=(\cat B')\fp$ and they generate the same localizing Serre subcategory: $\cat B=\cat B'$. See Remark~\ref{rem:BJ}.
\end{proof}

\begin{Rem}
\label{pt:pf-of-lattice-bijection}%
We are ready to prove Theorem~\ref{thm:lattices}. At this stage, all the maps in Theorem~\ref{thm:lattices} have been constructed, namely the announced inverse $[F]_{\simeq}\mapsto \Ker(F\otimes\yoneda)$ is well-defined by composing Theorems~\ref{thm:F-to-Serre} and~\ref{thm:B-to-smash}, as in the following picture:
\[
\kern-.5em
\xymatrix@C=4em@R=2em{
\kern-1em\left\{{{\displaystyle\textrm{smashing $\otimes$-ideals}}
 \atop{\vphantom{I^{I^I}}\displaystyle \cat S\subseteq\cat T}}\right\}
   \ar[d]_-{\cong}^-{\textrm{\cite{BalmerFavi11}}}
&
\left\{{{\displaystyle\textrm{localiz.\,$\otimes$-ideals $\cat B\subseteq\MT$ s.t.\ }\cat B=(\cat B\fp)^\to}
 \atop{\vphantom{I^{I^I}}\displaystyle \textrm{and s.t.\ }\calJ=\SET{f\textrm{ in }\cat T^c}{\img(\hat f)\in\cat B}\textrm{ equals }\calJ^2}}\right\}
 \ar[l]^-{\textrm{(Thm.\,\ref{thm:B-to-smash})}}_-{\yoneda\inv(\cat B)\, \mapsfrom \,\cat B}
\\
\left\{{{\displaystyle\textrm{right-idempotents }}
 \atop{\vphantom{I^{I^I}}\displaystyle \unit \to \bbf \ \textrm{ in }\cat T}}\right\}_{\!\!/\simeq}
 \ar[r]_-{[\bbf]_\simeq\,\mapsto \,[\hat \bbf]_\simeq}^-{\yoneda}
&
\left\{{{\displaystyle\textrm{flat right-idempotents }}
 \atop{\vphantom{I^{I^I}}\displaystyle\unit \to F \textrm{ in }\MT}}\right\}_{\!\!/\simeq}
  \ar[u]_-{\textrm{(Thm.\,\ref{thm:F-to-Serre})}}^-{[F]_\simeq\,\mapsto \,\Ker(F\otimes-)}
}
\]
The right-hand map is injective by Corollary~\ref{cor:idemp-eq-ab} (see Remark~\ref{rem:BF-inj}). The map at the top is injective by Corollary~\ref{cor:B-to-smash-inj}. Finally, let $\cat S\subseteq\cat T$ be a smashing $\otimes$-ideal with right-idempotent~$\bbf$, we have
\[
\cat S=\Ker(\bbf\otimes-)=\yoneda\inv\Ker(\hat\bbf\otimes-)
\]
by conservativity of restricted-Yoneda; see~\ref{pt:remind-MT}. This shows that the bottom map in the above diagram is injective and that the full-circle composition, from top-left to top-left, is the identity. This easily implies that all maps are surjective as well, thus finishing the proof of Theorem~\ref{thm:lattices}.
\qed
\end{Rem}

\goodbreak
\section{Distributivity of the lattice of smashing subcategories}
\label{se:distribute}%
\medbreak

%
\begin{Rem}
It already follows from~\cite{Krause05} that the lattice of smashing subcategories of a compactly generated triangulated category~$\cat T$ is complete. In the case where $\cat T$ is a \emph{tensor} triangulated category, rigidly-compactly generated of course, we verify that the \emph{$\otimes$-ideal} smashing subcategories form a complete sublattice, which then turns out to be a frame. We use Theorem~\ref{thm:lattices} to transpose the question to~$\MT$. The abelian setting is particularly nice since it provides arbitrary colimits, whereas the triangulated setting lacks non-countable homotopy colimits.
\end{Rem}

\begin{Prop}\label{prop:complete}
The lattice of isomorphism classes of flat right-idempotents in~$\cat A=\MT$ is complete. The join of a family $\{F_i\}_{i\in I}$ of flat right-idempotents is given by $\colim_{J\subseteq I}F_J$ in~$\cat A$, where $J$ runs over the poset of finite subsets of~$I$ and $F_J=\join_{j\in J}F_j\cong\otimes_{j\in J}F_j$ (see Proposition~\ref{prop:join}).
\end{Prop}

\begin{proof}
More precisely, when $J\subseteq J'$ are finite subsets of~$I$, we
have $F_J\le F_{J'}$. By Proposition~\ref{prop:idemp<}, there exists a
\emph{unique} morphism $\varphi_{J'J}:F_J\to F_{J'}$ such that
$\varphi_{J'J}\eta_{J}=\eta_{J'}$. This is the way we define the
functor from the poset of finite subsets of~$I$ to~$\cat A$. A priori,
the object $F_J$ is only defined up to isomorphism but we choose one
such object~$F_J$ for every~$J$ and then the morphisms are
forced. Changing this choice changes the functor $J\mapsto F_J$ up to
a unique isomorphism, hence yields isomorphic colimits. Since the
colimit is filtered, it is easy to show that $F:=\colim_{J\subseteq
  I}F_J$ remains flat (see~\ref{pt:coh}) and since $\otimes$ commutes
with colimits $F$ remains idempotent. For $i\in I$, we have $F_i\le
F_J$ for all $J\ni i$. Hence $F_J\otimes \eta_i$ is an isomorphism for
all such~$J$. As the~$J$ containing~$i$ are cofinal among finite
subsets~$J\subseteq I$, we get that $F\otimes \eta_i$ is an
isomorphism, hence $F_i\le F$. Conversely, if~$F_i\le F'$ for some
flat right-idempotent $\eta':\unit \to F'$ and all $i\in I$, we have
that $F'\otimes \varphi_{J'J}$ is an isomorphism for all $J\subseteq
J'$, since the morphism $\varphi_{J'J}:F_J\to F_{J'}$ consists of a
tensor of $\eta_j$ for $j\in J'\oursetminus J$. In other words, the
functor $J\mapsto F'\otimes F_J$ is isomorphic to the constant
functor~$J\mapsto F'$ and therefore, interchanging again the
  colimit and the tensor, $F'\otimes F\simeq F'$, which shows that $F\le F'$ as wanted.
\end{proof}

\begin{Lem}\label{lem:meet}
Let $\eta_1: \unit\to F_1$ and $\eta_2: \unit \to F_2$ be two flat idempotents in~$\cat A=\MT$. Let $\eta_{12}:=\eta_1 \otimes \eta_2 :\unit \to F_1\otimes F_2$ be their join (Proposition~\ref{prop:join}). Let $\eta_{12,1}:= F_1\otimes \eta_2 :F_1\to F_1\otimes F_2$ and $\eta_{12,2}:= \eta_1\otimes F_2 :F_2 \to F_1\otimes F_2$ denote the comparison morphisms giving $F_1 \le F_1\otimes F_2$ and $F_2\le F_1\otimes F_2$. Suppose that we have a flat idempotent $\eta:\unit\to F$ fitting in some 3-periodic exact sequence
\begin{equation}\label{eq:meet}
\vcenter{\xymatrix@C=1.5em{
& \cdots \ar[r]^-{\Sigma\inv\omega}
& F \ar[rr]^-{\smat{\varphi_1\\-\varphi_2}}
&& F_1 \oplus F_2 \ar[rrr]^-{\smat{\eta_{12,1} & \eta_{12,2}}}
&&& F_1 \otimes F_2 \ar[r]^-{\omega}
& \Sigma F \ar[r]
& \cdots
}}
\end{equation}
for some morphisms $\varphi_1$ and $\varphi_2$ and $\omega$. Then $F\simeq F_1 \meet F_2$ is the meet (infimum) of~$F_1$ and $F_2$.
\end{Lem}

\begin{proof}
Tensoring the sequence with the flat $F_1$, using that $F_1 \otimes \eta_{12,2}$ is an isomorphism, we deduce that $F_1\otimes\omega=0$. We thus get a split short exact sequence which shows that $F_1\otimes \varphi_1$ is an isomorphism and therefore $F\le F_1$. Similarly, we get $F\le F_2$. Suppose now that $F'$ is another flat right-idempotent such that $F'\le F_1$ and $F'\le F_2$. Tensoring the above sequence with $F'$ and using $\eta'$ to compare the above sequence to the new one, we have isomorphisms `at' $F_1\oplus F_2$ and $F_1\otimes F_2$ since $F_i\otimes \eta':F_i\isoto F_i\otimes F'$ is an isomorphism, for $i=1,2$. We then conclude by the Five Lemma that $F\otimes \eta'$ is also an isomorphism, hence $F'\le F$ as wanted.
\end{proof}

\begin{Thm}\label{thm:frame}
The lattice of isomorphism classes of flat right-idempotents in~$\cat A=\MT$ is \emph{infinitely distributive}, that is, for any (possibly infinite) family $\{F_i\}_{i\in I}$ and any $F$ we have
\[
F\meet (\bigjoin_{i\in I}F_i) \simeq \bigjoin_{i\in I} (F \meet F_i)\,.
\]
In other words, the lattice of smashing $\otimes$-ideals of~$\cat T$ is a frame.
\end{Thm}

\begin{proof}
In~$\cat T$, given a pair of right-idempotents $\bbf_1$ and $\bbf_2$, we have so-called Mayer-Vietoris distinguished triangles~\cite[Thm.\,3.13]{BalmerFavi11}
\begin{equation}
\label{eq:MV}%
\vcenter{\xymatrix@C=1.5em{
\bbf_1\meet \bbf_2 \ar[rr]^-{\smat{\varphi\\-\varphi}}
&& \bbf_1 \oplus \bbf_2 \ar[rr]^-{\smat{\varphi &\  \varphi}}
&& \bbf_1 \otimes \bbf_2 \ar[r]^-{}
& \Sigma \,\bbf_1\meet \bbf_2
}}
\end{equation}
where the morphisms denoted by a generic `$\varphi$' are (the unique) morphisms of idempotents.
This construction is natural in the idempotents~$\bbf_i$ in the sense that given any morphism of idempotents $\varphi_i:\bbf_i\to \bbf_i'$ there exists a \emph{unique} morphism of Mayer-Vietoris triangles which objectwise consists of: a morphism of idempotents $\varphi:\bbf_1\meet \bbf_2\to \bbf_1'\meet \bbf_2'$, the given $\smat{\varphi_1&0\\0&\varphi_2}:\bbf_1\oplus\bbf_2\to \bbf_1'\oplus \bbf_2'$ and their tensor $\varphi_1\otimes \varphi_2:\bbf_1\otimes\bbf_2\to \bbf_1'\otimes\bbf_2'$. Uniqueness refers to~$\varphi$. This result follows easily from the forced nature of~$\varphi$, as in Proposition~\ref{prop:idemp<}. Such morphisms will be referred to as \emph{morphisms of Mayer-Vietoris triangles}. The image of a Mayer-Vietoris triangle~\eqref{eq:MV} in~$\cat A=\MT$ under restricted-Yoneda $\yoneda:\cat T\to \cat A$ is a 3-periodic exact sequence as in~\eqref{eq:meet} above. Since we proved in Theorem~\ref{thm:lattices} that every flat right-idempotent in~$\cat A$ comes from~$\cat T$, we have such an exact sequence in~$\cat A$ for any flat right-idempotents $F_1$, $F_2$ in~$\cat A$:
\begin{equation}\label{eq:meet-again}
\vcenter{\xymatrix@C=1.5em{
& \cdots \ar[r]^-{}
& F_1\meet F_2 \ar[rr]^-{\smat{\varphi\\-\varphi}}
&& F_1 \oplus F_2 \ar[rr]^-{\smat{\varphi &\  \varphi}}
&& F_1 \otimes F_2 \ar[r]^-{}
& \Sigma F_1\meet F_2 \ar[r]
& \cdots
}}\kern-1em
\end{equation}
in which the morphisms denoted by~$\varphi$ are again (the unique) morphisms of idempotents. Lemma~\ref{lem:meet} provides a converse\,: the existence of such a sequence forces the first idempotent to be~$F_1\meet F_2$. Tensoring the above sequence~\eqref{eq:meet-again} with an arbitrary~$F$ thus shows that $F\otimes (F_1\meet F_2)\simeq (F\otimes F_1)\meet(F\otimes F_2)$, hence
\[
F\join (F_1\meet F_2)\simeq (F\join F_1)\meet (F\join F_2)\,.
\]
It is a general fact of lattice theory that this is equivalent to its dual:
\[
F\meet (F_1\join F_2)\simeq (F\meet F_1)\join (F\meet F_2)\,.
\]
A simple induction then implies that for every \emph{finite} $J\subseteq I$ we have
\begin{equation}\label{eq:aux-1}
F\meet \bigjoin_{j\in J} F_j \simeq \bigjoin_{j\in J} (F\meet F_j)\,.
\end{equation}
On the other hand, for every finite $J\subseteq I$, we have from~\eqref{eq:meet-again} an exact sequence:
\[
\xymatrix@C=1.5em{
& \cdots \ar[r]^-{}
& F\meet F_{J} \ar[rr]^-{\smat{\varphi\\-\varphi}}
&& F \oplus F_{J} \ar[rr]^-{\smat{\varphi &\  \varphi}}
&& F \otimes F_{J} \ar[r]^-{}
& \Sigma F\meet F_{J} \ar[r]
& \cdots
}
\]
in which the morphisms $\varphi$ are morphisms of idempotents. By the comments after~\eqref{eq:MV}, for every $J\subseteq J'$ the morphism of idempotent $\varphi_{J'J}:F_J\to F_{J'}$ gives rises to a morphism of Mayer-Vietoris triangles and therefore a morphism of exact sequences:
\[
\xymatrix@C=1.5em{
& \cdots \ar[r]^-{}
& F\meet F_{J} \ar[rr]^-{\smat{\varphi\\-\varphi}} \ar[d]_-{\exists}^-{\varphi}
&& F \oplus F_{J} \ar[rr]^-{\smat{\varphi &\  \varphi}} \ar[d]^-{\smat{1&0\\0&\varphi_{J'J}}}
&& F \otimes F_{J} \ar[r] \ar[d]^-{F\otimes \varphi_{J'J}}
& \Sigma F\meet F_{J} \ar[r]
& \cdots
\\
& \cdots \ar[r]^-{}
& F\meet F_{J'} \ar[rr]^-{\smat{\varphi\\-\varphi}}
&& F \oplus F_{J'} \ar[rr]^-{\smat{\varphi &\  \varphi}}
&& F \otimes F_{J'} \ar[r]^-{}
& \Sigma F\meet F_{J'} \ar[r]
& \cdots
}
\]
This defines a filtered system of exact sequences, indexed over the poset of finite subsets~$J\subseteq I$, by uniqueness of morphisms of idempotents (Proposition~\ref{prop:idemp<}). As filtered colimits are exact in~$\cat A$, we get a 3-periodic exact sequence:
\[
\xymatrix@C=1.3em{
\cdots \ar[r]^-{}
& \colim_{J} F\meet F_{J} \ar[r]^-{}
& F \oplus \bigvee_{i\in I} F_i \ar[rr]^-{\smat{\varphi &\  \varphi}}
&& F \otimes \bigvee_{i\in I} F_i \ar[r]^-{}
& \Sigma \colim_{J} F\meet F_{J}  \cdots
}
\]
Lemma~\ref{lem:meet} implies that $\colim_{J} F\meet F_{J}$ must be $F\meet \bigjoin_{i\in I}F_i$. Combining with the finite result in~\eqref{eq:aux-1}, we have
\[
F\meet\bigjoin_{i\in I} F_i \simeq
\colim_{J} F\meet F_J =
\colim_{J} F\meet \bigjoin_{j\in J} F_j \simeq
\colim_{J} \bigjoin_{j\in J} F\meet F_j \simeq
\bigjoin_{i\in I} F\meet F_i
\]
where the last isomorphism uses Proposition~\ref{prop:complete} for the family $\{F\meet F_i\}_{i\in I}$.
\end{proof}

\begin{Rem}
The above proof passed through the statement that $F\join (F_1\meet F_2)\simeq (F\join F_1)\meet (F\join F_2)$, which is equivalent to its dual. A colimit argument then showed the announced
$F\meet\bigjoin_{i\in I} F_i \simeq \bigjoin_{i\in I} F\meet F_i$. Note however that the latter is \emph{not} equivalent to its dual anymore. Similarly, our colimit argument cannot be replaced by a limit argument because (cofiltered) limits in~$\cat A$ are not exact.
\end{Rem}

\begin{Rem}
\label{rem:no-tensor-no-frame}%
If we do not restrict to ideals, \ie we consider all smashing subcategories of~$\cat T$ or all thick subcategories of~$\cat T^c$, then we may not get a frame; distributivity can fail even for finite joins. For instance, one can fix a field $k$ and consider $\bbP^1_k$, the projective line over $k$. For each $i\in \bbZ$ the line bundle $\mathcal{O}(i)$ on $\bbP^1_k$ generates a rather small thick subcategory
\begin{displaymath}
\cat L(i) := \mathrm{thick}(\mathcal{O}(i)) = \mathrm{add}(\Sigma^n\mathcal{O}(i) \; \vert \; n\in \bbZ) \cong \cat D^\mathrm{perf}(k).
\end{displaymath}
Choosing three distinct integers $i,j,$ and $k$ one can check that
\begin{displaymath}
\cat L(i) \meet (\cat L(j) \join \cat L(k)) = \cat L(i) \quad \text{but} \quad (\cat L(i) \meet \cat L(j)) \join (\cat L(i) \meet \cat L(k)) = 0.
\end{displaymath}
Of course we can inflate these thick subcategories to smashing subcategories to see that the lattice of smashing subcategories is not distributive either. This behavior should not be viewed as pathological, but rather as a common feature when the tensor unit does not generate.
\end{Rem}

\begin{Rem}
One could also consider the larger \emph{Bousfield lattice}, $A(\cat T)$, which consists of the localizing ideals arising as kernels of tensoring with the various objects of $\cat T$, \ie
\begin{displaymath}
A(\cat T) = \SET{\ker(X \otimes -)}{X\in \cat T}.
\end{displaymath}
The reader should be warned that the customary ordering on $A(\cat T)$ is by \emph{reverse inclusion}, which is the opposite of our convention for right-idempotents.

The Bousfield lattice has been studied in various contexts, especially for the stable homotopy category~$\SH$. For instance in~\cite{HoveyPalmieri99} (see also~\cite{Bousfield79b}) it is shown that $A(\SH)$ is not a frame, but it admits a subposet which is a frame and contains the smashing localizations. However, this is not directly comparable to our result as they take for their join operation (the meet with our ordering) the intersection of kernels, and as we shall see in the next remark this operation does not always preserve smashing subcategories.
\end{Rem}

\begin{Rem}
One can verify that under the bijection of Theorem~\ref{thm:lattices} the join constructed in Proposition~\ref{prop:complete} corresponds to the following construction on smashing $\otimes$-ideals: If $\{\cat S_i \;\vert \; i \in I\}$ is a family of smashing subcategories then their join is the localizing subcategory generated by their union
\begin{displaymath}
\bigvee_{i \in I} \cat S_i = \Loc (\cat S_i \;\vert \; i \in I).
\end{displaymath}
Finite meets are also relatively easy: $\cat S_1 \meet \cat S_2 = \cat S_1 \cap \cat S_2$. See~\cite[Prop.\,3.11]{BalmerFavi11}.

As an aside, we note that infinite meets can be somewhat different beasts, and cannot be computed by mere intersection, as the following example shows.

Let us fix a prime $p$ and consider the $p$-local stable homotopy category $\cat T=\SH_{(p)}$. Consider the family of those smashing subcategories of~$\SH_{(p)}$ which contain the following localizing subcategory
\begin{displaymath}
\cat D = \Ker\big(\coprod_{i\geq 0} K(i) \otimes -\big)
\end{displaymath}
of so-called \emph{dissonant spectra}, where $K(i)$ are the Morava $K$-theories at~$p$. The corresponding quotient $\cat H=\SH_{(p)}/\cat D$ is the category of \emph{harmonic spectra}. It follows from the analysis of~$\cat H$ in~\cite{Wolcott15} that the smashing subcategories of~$\SH_{(p)}$ which contain~$\cat D$ are precisely the \emph{finite} ones, and that their intersection is~$\cat D$, which is not smashing. Therefore the infinum of this family is not its intersection.
\end{Rem}

\begin{Rem}
\label{rem:conclusion}%
The reader might want to evaluate the novelty of our results, compared to the pre-existing theory developed without the tensor in~\cite{Krause00} and~\cite{Krause05}. Let us say a word about this in conclusion.

First, and rather obviously, the lattice isomorphism of Theorem~\ref{thm:lattices} relies on the tensor in its very statement: We need $\MT$ to have a tensor to even \emph{speak} of flat right-idempotents. Secondly, our proofs use the tensor in several critical places, actually from the very start with the Idempotence Lemma~\ref{lem:I^2} and its tensor-heavy proof. Also, when we prove in Section~\ref{se:smash} that the subcategory~$\cat S\subseteq\cat T$ is smashing, we use an argument which is different from the one in~\cite{Krause05}; an important aspect is precisely the Idempotence Lemma~\ref{lem:I^2} which allows us to prove idempotence of the ideal of morphisms~$\calJ$ upfront, whereas Krause obtains this idempotence \textsl{a posteriori}, see~\cite[Cor.\,12.6\,(3)]{Krause05}. Finally, we use the tensor to prove distributivity in Theorem~\ref{thm:frame}. We repeat that the latter result is generally false without the tensor. See Remark~\ref{rem:no-tensor-no-frame}. In short, our use of the tensor is essential and allows us to give a lean and quick treatment, as self-contained as possible.

This being said, the pre-tensor theory of~\cite{Krause05} involves several concepts that some readers might be familiar with, and which are interesting in themselves, like \emph{perfect} Serre subcategories, \emph{cohomological quotients}, \emph{flat epimorphisms} or \emph{exact quotients} of categories, \emph{cohomological} or \emph{exact} ideals of morphisms in~$\cat T^c$, and last but not least, \emph{pure-injective} objects in~$\cat T$. We do not use any of those here, although we indicate in Remark~\ref{rem:BJ} why the ideal $\calJ_F$ is cohomological and the subcategory $\cat B_F\fp$ is perfect.
\end{Rem}

\appendix
\goodbreak
\section{The module category~$\MT$ and its tensor structure}
\label{app:MT}%
\medbreak

The goal of this appendix is to provide a one-stop source of information on the module category~$\MT$. This material is mostly well-known, except for some of the tensorial aspects.

\begin{Rem}
\label{pt:remind-big-tt}%
We assume that the triangulated category~$\cat T$ admits arbitrary small coproducts; an object~$x\in\cat T$ is called \emph{compact} if~$\Homcat{T}(x,-)$ commutes with coproducts and $\cat T^c$ denotes the full thick triangulated subcategory of~$\cat T$ consisting of compact objects; a triangulated subcategory $\cat L\subseteq\cat T$ is called \emph{localizing} (respectively \emph{strictly localizing}) if it is closed under arbitrary coproducts (respectively the quotient $\cat T\too\cat T/\cat L$ admits a right adjoint); one assumes $\cat T$ to be \emph{compactly generated} meaning that $\cat T^c$ is essentially small and that $\cat T=\Loc(\cat T^c)$ is the smallest localizing subcategory containing~$\cat T^c$. A strictly localizing subcategory~$\cat S\subseteq\cat T$ is called \emph{smashing} if the right adjoint $\cat T/\cat S\too \cat T$ commutes with coproducts; examples arise as $\cat S=\Loc(\cat J)$ the localizing subcategory generated by a subcategory $\cat J\subseteq\cat T^c$ of compact objects; one says that~$\cat T$ satisfies the \emph{telescope property} if every smashing subcategory~$\cat S$ of~$\cat T$ is of this form: $\cat S=\Loc(\cat S^c)$.

We further assume that $\cat T$ is closed symmetric monoidal, in a compatible way with the triangulation. Essentially, the tensor~$\otimes$ is exact in each variable and preserves coproducts. See details in~\cite[App.\,A]{HoveyPalmieriStrickland97}. We also assume that compact and rigid objects coincide. In words, this is called a \emph{rigidly-compactly generated tensor-triangulated category} (or sometimes colloquially a `big tt-category', the `small' one referring to~$\cat T^c$). These are the categories used in~\cite{BalmerFavi11} for instance. There, it is proved that every smashing $\otimes$-ideal~$\cat S$ in~$\cat T$ is characterized by the isomorphism class of the corresponding idempotent triangle~$\bbe\oto{\eps}\unit\oto{\eta} \bbf\to\Sigma \bbe$, with `idempotent' meaning that $\bbe\oto{\eps}\unit$ is a left-idempotent and $\unit\oto{\eta}\bbf$ is a right-idempotent, which in short reads $\bbe\otimes\bbf=0$.
The smashing $\otimes$-ideal~$\cat S$ corresponding to this triangle is given by $\cat S=\Ker(\bbf\otimes-)$, which is also (the replete-closure of)~$\bbe\otimes\cat T$.
\end{Rem}

\begin{Rem}
\label{pt:remind-MT}%
The category $\MT$ is the category of contravariant additive functors from~$\cat T^c$ to the category~$\Ab$ of abelian groups. It is well-known to be a Grothendieck category (see~\cite[\S\,5.10]{BucurDeleanu68} if necessary) whose finitely presented objects are denoted by~$\mT$. The \emph{restricted-Yoneda functor} $\yoneda:\cat T\to \MT$ maps $X\in \cat T$ to $\yoneda(X)=\hat X$ defined as follows:
\[
\hat X:({\cat T^c})^\mathrm{op}\to \Ab, \quad \hat X(-)= \Homcat{T}(-,X).
\]
The functor $\yoneda:\cat T\to \MT$ is conservative since $\cat T$ is generated by~$\cat T^c$. It is also \emph{homological} (\ie maps distinguished triangles to exact sequences) and preserves arbitrary coproducts. We return to this in~\ref{pt:hatF} below. We have the following commutative diagram
\[
\xymatrix{
\cat T^c \vphantom{I_{I^I}}\  \ar@{^(->}[r]^-{\yoneda} \ar@{^(->}[d]
& \mT \vphantom{I_{I^I}} \ar@{^(->}[d]
 \ar@{}[r]|-{=}
& \cat A\fp\\
\cat T \ar[r]^-{\yoneda}
& \MT \ar@{}[r]|-{=:}
& \cat A\,.
}
\]
The light notation $\hat X$, as opposed to~$\yoneda(X)$, allows us to think of the object $\hat X$ as being \emph{almost}~$X$ itself, although $\yoneda:\cat T\to \MT$ is in general neither full nor faithful outside of~$\cat T^c$. Note however that a simple application of Yoneda's Lemma guarantees that $\yoneda$ induces an isomorphism
\begin{equation}
\label{eq:yoneda-hom}%
\Homcat{T}(x,Y)\isoto \Homcat{A}(\hat x,\hat Y)
\end{equation}
for compact~$x\in\cat T^c$ and arbitrary~$Y\in\cat T$, since both sides of~\eqref{eq:yoneda-hom} are~$\hat Y(x)$. This remains true if we replace~$x$ by a coproduct $\coprod_{i}x_i$ in~$\cat T$ with all~$x_i\in\cat T^c$.
\end{Rem}

\begin{Rem}
\label{pt:sigma}%
Suspension~$\Sigma:\cat T\to \cat T$ induces a suspension~$\Sigma:\cat A\isoto \cat A$ on the module category~$\cat A=\MT$, which is defined by~$\Sigma M(x)=M(\Sigma\inv(x))$. This is done in such a way that $\yoneda:\cat T\to \cat A$ is compatible with suspensions: $\Sigma\circ\yoneda\cong\yoneda\circ\Sigma$.
\end{Rem}

\begin{Rem}
\label{pt:mT}%
Limits and colimits are computed objectwise in the functor category~$\cat A=\MT$. It follows from~\eqref{eq:yoneda-hom} that $\hat x$ is projective and finitely presented in~$\MT$ for every~$x\in\cat T^c$. In fact, the subcategory~$\mT$ is equivalent to the \emph{Freyd envelope of~$\cat{T}^c$}~\cite[Chap.\,5]{Neeman01}, which is Frobenius abelian. The Yoneda embedding $\cat T^c\hook \mT$ identifies~$\cat T^c$ with the projective and injective objects in~$\mT$. Although these objects $\hat x$ remain projective they are usually \emph{not} injective in the whole module category $\MT$. In fact, the functor $\yoneda:\cat T\to \cat A$ restricts to an equivalence between the so-called \emph{pure-injective} objects in~$\cat T$ and the injective objects of~$\MT$. See~\cite{Krause00}. (We shall not use the latter correspondence here.) It is nevertheless useful to be able to embed any finitely presented $m\in\mT$ as $m\into \hat y$ for some $y\in\cat T^c$, or to write $m=\img(\hat f)$ for some $f:x\to y$ in~$\cat T^c$. In particular, any (replete) subcategory~$\cat B_0$ of~$\cat A\fp$ is characterized by the class of morphisms~$f$ in~$\cat T^c$ whose image belongs to~$\cat B_0$. For big modules, we have:
\end{Rem}

\begin{Prop}
\label{prop:enough}%
Let $M$ be in~$\MT$. Then there exists a distinguished triangle $X\oto{f}Y\oto{g} Z\oto{h}\Sigma X$ in~$\cat T$ such that $X$ and $Y$ are (small) coproducts of compact objects and $M\simeq\img(\hat g)$. In particular, there exist an epimorphism $\hat Y\onto M$, a monomorphism $M\into \hat Z$ and an exact sequence
\[
0\to M\to \hat Z \to \Sigma \hat X \to \Sigma \hat Y \to \Sigma M\to 0\,.
\]
\end{Prop}

\begin{proof}
Every $M\in\MT$ is a quotient $\coprod_{i\in I}\hat y_i\onto M$ of a coproduct of finitely presented projectives~$\hat y_i$ with $y_i\in\cat T^c$ (take $I$ the disjoint union of the sets~$M(y)$ where $y$ runs over a skeleton of the essentially small~$\cat T^c$). This coproduct $\coprod_{i}\hat y_i$ is an object~$\hat Y$ for $Y=\coprod_{i}y_i$ the corresponding coproduct of compacts in~$\cat T$. Repeating the argument, $M$ has a presentation $\hat X\oto{\hat f} \hat Y\onto M$ for $X,Y\in\cat T$ coproducts of compacts; indeed any map $\hat X\to \hat Y$ must come from~$\cat T$ by~\eqref{eq:yoneda-hom} and the nature of~$X$. Completing $f$ into a distinguished triangle and using that $\yoneda:\cat T\to \MT$ is homological, we obtain the result.
\end{proof}

\begin{Rem}
\label{pt:coh}
Our abelian category $\cat A=\MT$ of modules is a \emph{locally coherent} (Grothen\-dieck) category. Let us remind the reader about `locally coherent'. First, $\cat A$ has small colimits (coproducts) and filtered colimits are exact. An object $m\in\cat A$ is called \emph{finitely presented} if $\Homcat{A}(m,-)$ commutes with filtered colimits. We write $\cat A\fp$ for the subcategory of finitely presented objects, which is equivalent to~$\mT$. So $\cat A\fp$ is essentially small and every object of~$\cat A$ is a filtered colimit of objects in~$\cat A\fp$. This is called \emph{locally finitely presented} and is equivalent to $\cat A$ being a Grothendieck category with a generating set of finitely presented objects. Finally, the subcategory $\cat A\fp=\mT$ is itself abelian, which makes~$\cat A$ a \emph{locally coherent} abelian category. See details in~\cite[\S\,1]{Krause97} and further references therein.
\end{Rem}

\begin{Rem}
\label{pt:B-loc}%
A Serre subcategory $\cat B\subseteq\cat A$ is \emph{localizing} if it is closed under arbitrary coproducts (or colimits), \ie the Gabriel quotient $\cat A\onto \cat A/\cat B$ admits a right adjoint. Giving a Serre subcategory $\cat B_0$ of $\cat A\fp$ is equivalent to giving the localizing subcategory $\cat B:=\cat B_0^\to$ it generates. Indeed, we can recover~$\cat B_0$ as the finitely presented objects of~$\cat B$, that is, $\cat B_0=\cat B\fp=\cat B\cap \cat A\fp$. The localizing category $\cat B_0^\to$ consists of all colimits of objects of~$\cat B_0$ and can also be described as those $M\in\cat A$ such that every morphism $m\to M$ with $m$ finitely presented factors via some object of~$\cat B_0$.

We shall prove below that $\cat A$ admits a tensor. Since $(\cat A\fp)^\to=\cat A$, it will then be clear that a Serre subcategory~$\cat B_0\subseteq\cat A\fp$ is \emph{$\otimes$-ideal} (\ie $\cat A\fp\otimes\cat B_0\subseteq\cat B_0$) if and only if the induced localizing Serre subcategory $\cat B=\cat B_0^\to$ is \emph{$\otimes$-ideal} in~$\cat A$ (\ie $\cat A\otimes\cat B\subseteq\cat B$).
\end{Rem}

\begin{Rem}
\label{rem:sigma}%
All subcategories of~$\cat A$ that we are considering are tacitly assumed stable under suspension~$\Sigma:\cat A\to \cat A$. See~\ref{pt:sigma}. For instance, we really only consider $\Sigma$-stable Serre subcategories. This condition will be subsumed anyway in the $\otimes$-ideal property since $\Sigma(-)\cong \Sigma(\unit)\otimes-$. We therefore omit it everywhere.
\end{Rem}

\begin{Rem}
\label{pt:hatF}%
The restricted-Yoneda functor $\yoneda:\cat T\to \MT$ is homological and preserves arbitrary coproducts. It is also universal for this property~\cite[Cor.\,2.4]{Krause00} in that any coproduct-preserving homological functor $F:\cat T\to \cat {A}$ to a Grothendieck category factors uniquely via~$\yoneda$:
\[
\xymatrix{
\cat T \ar[r]^-{\yoneda} \ar[rd]_-{F}
& \MT \ar[d]^-{\hat F}
\\
& \cat A
}
\]
through an \emph{exact} functor~$\hat F$ which is colimit-preserving (\ie coproduct preserving).
The extension $\hat F$ is characterized by the following formula, for every~$M\in\MT$
\begin{equation}
\label{eq:hatF}%
\hat F(M)=\colim_{(x\,,\,\hat x\oto{} M)\,\in\,[\cat T^c\oto{\yoneda}M]}F(x)
\end{equation}
where $[\cat T^c\oto{\yoneda} M]$ is the `slice category', sometimes denoted~$(\yoneda\to M)$ or~$(\yoneda\searrow M)$.
In fact, since $\hat F$ is colimit-preserving and since every $M\in\MT$ is the colimit
\[
M=\colim_{(x\,,\,\hat x\to M)\,\in\,[\cat T^c\oto{\yoneda}M]}\hat x
\]
we can simply characterize~$\hat F$ by $\hat F(\hat x)=F(x)$ for all compact~$x\in\cat T^c$.
\end{Rem}

\begin{Exa}
\label{ex:hatT}%
For every object~$X\in \cat T$, we can consider the homological coproduct-preserving functor
\[
\xymatrix{\cat T \ar[r]^-{X\otimes-} & \cat T \ar[r]^-{\yoneda} & \MT
}
\]
and thus obtain an exact colimit-preserving functor~$T^{\mathrm{left}}_X:\MT\to \MT$ such that $T^{\mathrm{left}}_X(\hat Y)=\widehat{X\otimes Y}$. Similarly, we get for each $Y\in \cat T$ an exact colimit-preserving functor $T^{\mathrm{right}}_Y:\MT\to \MT$ characterized by~$T^{\mathrm{right}}_Y(\hat X)=\widehat{X\otimes Y}$.
\end{Exa}

\begin{Rem}
We also have the \emph{Day convolution product} on~$\MT$
\[
\otimes: \MT \times \MT \too \MT
\]
which is only \emph{right exact} and colimit-preserving in each variable and described by the following formula, for every~$M,N\in\MT$
\[
M\otimes N=\colim_{(\hat x\to M)\in [\cat T^c\oto{\yoneda}M]}\quad \colim_{(\hat y\to N)\in [\cat T^c\oto{\yoneda}N]} \quad \widehat{x\otimes y}\,.
\]
We see right away that for $x,y\in\cat T^c$, we have agreement $\hat x\otimes \hat y=\widehat{x\otimes y}=T^{\mathrm{left}}_x(\hat y)=T^{\mathrm{right}}_y(\hat x)$. Taking a suitable colimit in the variable~$x$ (over the slice category of~$M$) in the equation $\hat x\otimes \hat y=T^{\mathrm{right}}_y(\hat x)$ we get that $M\otimes \hat y = T^{\mathrm{right}}_y(M)$ for every $M\in \MT$; in particular for $M=\hat X$
\[
\hat X \otimes \hat y = T^{\mathrm{right}}_y(\hat X) = \widehat{X\otimes y} = T^{\mathrm{left}}_X(\hat y)
\]
for every $X\in\cat T$. Taking now a similar colimit in~$y$ in the latter, we see that the functor $\hat X\otimes -$ coincides with the functor $T^{\mathrm{left}}_X$ of Example~\ref{ex:hatT} and in particular
\[
\hat X\otimes \hat Y\cong \widehat{X\otimes Y}\,.
\]
As $\hat X\otimes-$ coincides with the exact functor $T^{\mathrm{left}}_X$, we have also shown that each $\hat X$ is flat.
Finally the tensor restricts to finitely presented objects, namely:
\[
\cat A\fp \otimes \cat A\fp\subseteq\cat A\fp
\qquadtext{and}
\unit_{\cat A}=\hat\unit_{\cat T}\in\cat A\fp,
\]
as the usual Day convolution product on~$\cat A\fp=\mT$ induced by that on~$\cat T^c$.
\end{Rem}

In summary we have:

\begin{Prop}
\label{prop:tens-MT}%
The category~$\MT$ admits a right-exact tensor $\otimes$ commuting with colimits. Moreover, $\yoneda:\cat T\to \MT$ is monoidal: $\hat X\otimes \hat Y\cong \widehat{X\otimes Y}$ and the objects $\hat X\in \MT$ are flat, for all~$X\in\cat T$ (\ie $\hat X\otimes-$ is exact).
\qed
\end{Prop}

\begin{Rem}
In fact the category~$\MT$ is moreover closed, i.e.\ it admits an internal hom. This follows from a general argument about Grothendieck categories.
\end{Rem}

\goodbreak

\end{document}